\documentclass[12pt, reqno]{amsart}
\usepackage{amsmath, amsthm, amscd, amsfonts, amssymb, graphicx, color}
\usepackage[bookmarksnumbered, colorlinks, plainpages]{hyperref}

\newtheorem{theorem}{Theorem}[section]
\newtheorem{lemma}[theorem]{Lemma}

\newtheorem{corollary}[theorem]{Corollary}
\theoremstyle{definition}

\theoremstyle{remark}

\numberwithin{equation}{section}

\begin{document}
\setcounter{page}{1}

\title[Conservative algebras]{Conservative algebras of $2$-dimensional algebras, III}

\author[F. Arzikulov]{Farhodjon Arzikulov$^{1,2}$}

\address{$^{1}$ V.I. Romanovskiy Institute of Mathematics, Namangan Regional Department, Uzbekistan Academy of Sciences}
\address{$^{2}$ Department of Mathematics, Andizhan State University, Andizhan, Uzbekistan.}
\email{\textcolor[rgb]{0.00,0.00,0.84}{arzikulovfn@rambler.ru}}

\author[N. Umrzaqov]{Nodirbek Umrzaqov$^3$}

\address{$^{3}$ Department of Mathematics, Andizhan State University, Andizhan, Uzbekistan.}
\email{\textcolor[rgb]{0.00,0.00,0.84}{umrzaqov2010@mail.ru}}


\begin{abstract}
In the present paper we prove that every local and $2$-local
derivation on conservative algebras of $2$-dimensional algebras are derivations.
Also, we prove that every local and $2$-local
automorphism on conservative algebras of $2$-dimensional algebras are automorphisms.
\end{abstract}


\maketitle

\footnotetext{}

\section{Introduction}

The present paper is devoted to the study of conservative algebras.
In 1972 Kantor \cite{KI2} introduced conservative algebras as a generalization of Jordan algebras
(also, see a good written survey about the study of conservative algebras \cite{pp20}).

In 1990 Kantor \cite{KI} defined the multiplication $\cdot$ on the set of all algebras
(i.e. all multiplications) on the $n$-dimensional vector space $V_n$
over a field ${ \mathbb F}$ of characteristic zero as follows: $A\cdot B=[L^A_e,B]$,
where $A$ and $B$ are multiplications and $e\in V_n$ is some fixed vector. If $n>1$, then the algebra
$W(n)$ does not belong to any well-known class of algebras (such as associative, Lie, Jordan, or
Leibniz algebras). The algebra $W(n)$ is a conservative algebra \cite{KI2}.

In \cite{KI2} Kantor classified all conservative $2$-dimensional algebras and defined the class of terminal algebras as algebras
satisfying some certain identity. He proved that every terminal algebra is a conservative algebra
and classified all simple finite-dimensional terminal algebras with left quasi-unit over an
algebraically closed field of characteristic zero \cite{KI3}. Terminal algebras were also studied in \cite{KKS,KKP}.

In 2017 Kaygorodov and Volkov \cite{KV} described automorphisms, one-sided ideals, and idempotents of $W(2)$. Also a similar problem
is solved for the algebra $W_2$ of all commutative algebras on the $2$-dimensional vector space
and for the algebra $S_2$ of all commutative algebras with zero multiplication trace on the
$2$-dimensional vector space. The papers \cite{KLP,KPP} are also devoted to the study of conservative
algebras and superalgebras.

Let $\mathcal{A}$ be an algebra. A linear operator $\nabla$ on $\mathcal{A}$ is called a local
derivation if for every $x\in \mathcal{A}$ there exists a derivation $\phi_x$ of $\mathcal{A}$,
depending on $x$, such that $\nabla(x)=\phi_x(x)$.
The history of local derivations had begun from the paper of Kadison \cite{k90}.
Kadison introduced the concept of local derivation and proved
that each continuous local derivation from a von Neumann algebra into its dual Banach
bimodule is a derivation.

A similar notion, which characterizes nonlinear generalizations of derivations, was introduced by  \v{S}emrl as  $2$-local derivations.
In his paper \cite{S} was proved that a $2$-local derivation of the algebra
$B(H)$ of all bounded linear operators on the infinite-dimensional
separable Hilbert space $H$ is a derivation.
After his works, appear numerous new results related to the description of local and $2$-local derivations of associative algebras
(see, for example, \cite{AA2}, \cite{AK}, \cite{AK2}, \cite{kh19}, \cite{KK}, \cite{LW}).

The study of local and $2$-local derivations of non-associative algebras was initiated in some
papers of  Ayupov and  Kudaybergenov (for the case of Lie algebras, see \cite{AyupovKudaybergenov,ak16}).
In particular, they proved that there are no non-trivial local and $2$-local derivations on semisimple finite-dimensional Lie algebras.
In \cite{AyuKudRak16} examples of $2$-local derivations on nilpotent Lie algebras which are not derivations, were also given.
Later, the study of local and $2$-local derivations was continued for Leibniz algebras \cite{akoz}, Malcev algebras and Jordan algebras \cite{aa17}.
Local automorphisms and $2$-local automorphisms, also were studied in many cases,
for example, they were studied on Lie algebras \cite{AyupovKudaybergenov,Costantini}.

Now, a linear operator $\nabla$ on $\mathcal{A}$ is called a local
automorphism if for every $x\in \mathcal{A}$ there exists an automorphism $\phi_x$ of $\mathcal{A}$,
depending on $x$, such that $\nabla(x)=\phi_x(x)$. The concept of local automorphism was introduced
by Larson and Sourour \cite{LarsonSourour} in 1990. They proved that, invertible local automorphisms of the algebra
of all bounded linear operators on an infinite-dimensional Banach space $X$ are automorphisms.

A similar notion, which characterizes non-linear generalizations of automorphisms, was introduced
by  \v{S}emrl in \cite{S} as $2$-local automorphisms.
Namely, a map $\Delta :\mathcal{A}\to \mathcal{A}$ (not necessarily linear) is called a $2$-local automorphism, if
for every $x,y\in \mathcal{A}$ there exists an automorphism $\phi_{x,y}: \mathcal{A}\to \mathcal{A}$
such that $\Delta(x)=\phi_{x,y}(x)$ and $\Delta(y)=\phi_{x,y}(y)$.
After the work of \v{S}emrl, it appeared numerous new results related to the description of local and $2$-local automorphisms of algebras
(see, for example, \cite{AyupovKudaybergenov}, \cite{akoz}, \cite{ChenWang}, \cite{Costantini}, \cite{KK}).

In the present paper, we continue the study of derivations, local and $2$-local derivations of conservative algebras of $2$-dimensional algebras.
We prove that every local and  $2$-local derivation of the conservative algebras of $2$-dimensional algebras are derivations.
In the present paper, we continue the study of automorphisms, local and $2$-local automorphisms in the case of
conservative algebras of $2$-dimensional algebras. We prove that every local and  $2$-local automorphism of
the conservative algebras of $2$-dimensional algebras are automorphisms.

\section{Preliminaries}

Throughout this paper ${ \mathbb F}$ is some fixed field of characteristic zero.
A multiplication on $2$-dimensional vector space is defined by a $2\times 2\times 2$ matrix. Their
classification was given in many papers (see, for example, \cite{P}). Let consider the space $W(2)$ of all
multiplications on the $2$-dimensional space $V_2$ with a basis $v_1$, $v_2$. The definition of the
multiplication $\cdot$ on the algebra $W(2)$ is defined as follows: we fix the vector $v_1\in V_2$
and define
$$
(A\cdot B)(x,y)=A(v_1,B(x,y))-B(A(v_1,x),y)-B(x,A(v_1,y))
$$
for $x$, $y\in V_2$ and $A$, $B\in W(2)$. The algebra $W(2)$ is conservative \cite{KI}.
Let consider the multiplications $\alpha^k_{i,j}$ ($i$, $j$, $k=1,2$) on $V_2$ defined by the formula
$\alpha^k_{i,j}(v_t,v_l)=\delta_{it}\delta_{jl}v_k$ for all $t$, $l\in \{1,2\}$. It is easy to see that
$\{\alpha^k_{i,j}\vert i,j,k=1,2\}$ is a basis of the algebra $W(2)$. The
multiplication table of $W(2)$ in this basis is given in \cite{KLP}. In this work we use another basis for
the algebra $W(2)$ (from \cite{KV}). Let introduce the notation
$$
e_1=\alpha^1_{11}-\alpha^2_{12}-\alpha^2_{21}, \,\,
e_2=\alpha^2_{11}, \,\, e_3=\alpha^2_{22}-\alpha^1_{12}-\alpha^1_{21}, \,\,
e_4=\alpha^1_{22}, \,\, e_5=2\alpha^1_{11}+\alpha^2_{12}+\alpha^2_{21},
$$
$$
e_6=2a^2_{22}+\alpha^1_{12}+\alpha^1_{21}, \,\,
e_7=\alpha^1_{12}-\alpha^1_{21}, \,\, e_8=\alpha^2_{12}-\alpha^2_{21}.
$$
It is easy to see that the multiplication table of $W(2)$ in the basis $e_1,\dots,e_8$ is the following.

\begin{center}
\begin{tabular}{|c|c|c|c|c|c|c|c|c|}
  \hline
  % after \\: \hline or \cline{col1-col2} \cline{col3-col4} ...
   & $e_1$ & $e_2$ & $e_3$ & $e_4$ & $e_5$ & $e_6$ & $e_7$ & $e_8$ \\ \hline
  $e_1$ & $-e_1$ & $-3e_2$ & $e_3$ & $3e_4$ & $-e_5$ & $e_6$ & $e_7$ & $-e_8$ \\ \hline
  $e_2$ & $3e_2$ & $0$ & $2e_1$ & $e_3$ & $0$ & $-e_5$ & $e_8$ & $0$ \\ \hline
  $e_3$ & $-2e_3$ & $-e_1$ & $-3e_4$ & $0$ & $e_6$ & $0$ & $0$ & $-e_7$ \\ \hline
  $e_4$ & $0$ & $0$ & $0$ & $0$ & $0$ & $0$ & $0$ & $0$ \\ \hline
  $e_5$ & $-2e_1$ & $-3e_2$ & $-e_3$ & $0$ & $-2e_5$ & $-e_6$ & $-e_7$ & $-2e_8$ \\ \hline
  $e_6$ & $2e_3$ & $e_1$ & $3e_4$ & $0$ & $-e_6$ & $0$ & $0$ & $e_7$ \\ \hline
  $e_7$ & $2e_3$ & $e_1$ & $3e_4$ & $0$ & $-e_6$ & $0$ & $0$ & $e_7$ \\ \hline
  $e_8$ & $0$ & $e_2$ & $-e_3$ & $-2e_4$ & $0$ & $-e_6$ & $-e_7$ & $0$ \\ \hline
  %\hline
\end{tabular}
\end{center}

The subalgebra generated by the elements $e_1,\dots,e_6$ is the conservative (and, moreover,
terminal) algebra $W_2$ of commutative $2$-dimensional algebras. The subalgebra generated by the
 elements $e_1,\dots,e_4$ is the conservative (and, moreover, terminal) algebra $S_2$ of all
commutative $2$-dimensional algebras with zero multiplication trace \cite{KLP}.

Let $\mathcal{A}$ be an algebra.
A linear map $D : \mathcal{A}\to \mathcal{A}$ is called
a derivation, if $D(xy)=D(x)y+xD(y)$ for any two elements $x$, $y\in \mathcal{A}$.

Our main tool for study of local and $2$-local derivations of
the algebras $S_2$, $W_2$ and $W(2)$ is the following lemma \cite[Theorem 6]{KLP},
where the matrix of a derivation is calculated in the new basis $e_1,\dots,e_8$.

\begin{lemma} \label{3}
A linear map $D:W(2)\to W(2)$ is a
derivation if and only if the matrix of $D$ has the following matrix form:
\begin{equation} \label{(1.3)}
\left(
    \begin{array}{cccccccc}
      0 & \alpha & 0 & 0 & 0 & 0 & 0 & 0\\
      0 & -\beta & 0 & 0 & 0 & 0 & 0 & 0\\
      2\alpha & 0 & \beta & 0 & 0 & 0 & 0 & 0\\
      0 & 0 & 3\alpha & 2\beta & 0 & 0 & 0 & 0 \\
      0 & 0 & 0 & 0 & 0 & 0 & 0 & 0\\
      0 & 0 & 0 & 0 & -\alpha & \beta & 0 & 0\\
      0 & 0 & 0 & 0 & 0 & 0 & \beta & \alpha\\
      0 & 0 & 0 & 0 & 0 & 0 & 0 & 0\\
      \end{array}
  \right),
\end{equation}
where $\alpha$, $\beta$ are elements in ${ \mathbb F}$.
\end{lemma}

Now, we give a characterization of automorphisms on conservative algebras of $2$-dimensional algebras.

Let $\mathcal{A}$ be an algebra.
A bijective linear map $\phi : \mathcal{A}\to \mathcal{A}$ is called
an automorphism, if $\phi(xy)=\phi(x)\phi(y)$ for any elements $x$, $y\in \mathcal{A}$.

Our principal tool for study of local and $2$-local automorphisms of the algebras
$S_2$, $W_2$ and $W(2)$ is the following lemma, which was proved in \cite[Theorem 11]{KV}.

\begin{lemma} \label{21}
A linear map $\phi:W(2)\to W(2)$ is an
automorphism if and only if the matrix of $\phi$ has the following matrix form:
\begin{equation} 
\left(
    \begin{array}{cccccccc}
      1 & a & 0 & 0 & 0 & 0 & 0 & 0\\
      0 & \frac{1}{b} & 0 & 0 & 0 & 0 & 0 & 0\\
      2ab & a^2b & b & 0 & 0 & 0 & 0 & 0\\
      3a^2b^2 & a^3b^2 & 3ab^2 & b^2 & 0 & 0 & 0 & 0 \\
      0 & 0 & 0 & 0 & 1 & 0 & 0 & 0\\
      0 & 0 & 0 & 0 & -ab & b & 0 & 0\\
      0 & 0 & 0 & 0 & 0 & 0 & b & ab\\
      0 & 0 & 0 & 0 & 0 & 0 & 0 & 1\\
      \end{array}
  \right),
\end{equation}
where $a$, $b$ are elements in ${ \mathbb F}$ and $b\neq 0$.
\end{lemma}

\section{Local derivations of conservative algebras of $2$-dimensional algebras}

In this section we give a characterization of derivations on conservative algebras of $2$-dimensional algebras.

Let $\mathcal{A}$ be an algebra.
A linear map $\nabla : \mathcal{A}\to \mathcal{A}$ is called a local derivation, if for
any element $x\in \mathcal{A}$ there exists a derivation
$D_x:\mathcal{A}\to \mathcal{A}$ such that $\nabla(x)=D_x(x)$.

\begin{theorem}\label{5}
Every local derivation of the algebra $W(2)$ is a derivation.
\end{theorem}

\begin{proof}
Let $\nabla$ be an arbitrary local derivation of $W(2)$ and write
\[
\nabla(x)=B\bar{x}, x\in W(2),
\]
where $B=(b_{i,j})_{i,j=1}^8$, $\bar{x}=(x_1,x_2,x_3,x_4,x_5,x_6,x_7,x_8)$ is the vector corresponding to $x$.
Then for every $x\in W(2)$ there exist elements $a_x$, $b_x$ in ${ \mathbb F}$ such that
\[
B\bar{x}=
\left(
    \begin{array}{cccccccc}
      0 & a_x & 0 & 0 & 0 & 0 & 0 & 0\\
      0 & -b_x & 0 & 0 & 0 & 0 & 0 & 0\\
      2a_x & 0 & b_x & 0 & 0 & 0 & 0 & 0\\
      0 & 0 & 3a_x & 2b_x & 0 & 0 & 0 & 0 \\
      0 & 0 & 0 & 0 & 0 & 0 & 0 & 0\\
      0 & 0 & 0 & 0 & -a_x & b_x & 0 & 0\\
      0 & 0 & 0 & 0 & 0 & 0 & b_x & a_x\\
      0 & 0 & 0 & 0 & 0 & 0 & 0 & 0\\
      \end{array}
  \right)
\left(
  \begin{array}{c}
    x_1 \\
    x_2 \\
    x_3 \\
    x_4 \\
    x_5 \\
    x_6 \\
    x_7 \\
    x_8 \\
  \end{array}
\right).
\]
In other words
\[
\left\{
  \begin{array}{ll}
    b_{1,1}x_1+b_{1,2}x_2+b_{1,3}x_3+b_{1,4}x_4+b_{1,5}x_5+b_{1,6}x_6+b_{1,7}x_7+b_{1,8}x_8=a_xx_2 \hbox{;} \\
    b_{2,1}x_1+b_{2,2}x_2+b_{2,3}x_3+b_{2,4}x_4+b_{2,5}x_5+b_{2,6}x_6+b_{2,7}x_7+b_{2,8}x_8=-b_xx_2 \hbox{;} \\
    b_{3,1}x_1+b_{3,2}x_2+b_{3,3}x_3+b_{3,4}x_4+b_{3,5}x_5+b_{3,6}x_6+b_{3,7}x_7+b_{3,8}x_8=2a_xx_1+b_xx_3 \hbox{;} \\
    b_{4,1}x_1+b_{4,2}x_2+b_{4,3}x_3+b_{4,4}x_4+b_{4,5}x_5+b_{4,6}x_6+b_{4,7}x_7+b_{4,8}x_8=3a_xx_3+2b_xx_4 \hbox{;} \\
    b_{5,1}x_1+b_{5,2}x_2+b_{5,3}x_3+b_{5,4}x_4+b_{5,5}x_5+b_{5,6}x_6+b_{5,7}x_7+b_{5,8}x_8=0 \hbox{;} \\
    b_{6,1}x_1+b_{6,2}x_2+b_{6,3}x_3+b_{6,4}x_4+b_{6,5}x_5+b_{6,6}x_6+b_{6,7}x_7+b_{6,8}x_8=-a_xx_5+b_xx_6 \hbox{;} \\
    b_{7,1}x_1+b_{7,2}x_2+b_{7,3}x_3+b_{7,4}x_4+b_{7,5}x_5+b_{7,6}x_6+b_{7,7}x_7+b_{7,8}x_8=b_xx_7+a_xx_8 \hbox{;} \\
    b_{8,1}x_1+b_{8,2}x_2+b_{8,3}x_3+b_{8,4}x_4+b_{8,5}x_5+b_{8,6}x_6+b_{8,7}x_7+b_{8,8}x_8=0 \hbox{.}
  \end{array}
\right.
\]
Taking $x=(1,0,0,0,0,0,0,0)$, $x=(0,0,1,0,0,0,0,0)$, $x=(0,0,0,1,0,0,0,0)$, etc,
from this it follows that
\[
b_{1,1}=b_{1,3}=b_{1,4}=b_{1,5}=b_{1,6}=b_{1,7}=b_{1,8}=
\]
\[
=b_{2,1}=b_{2,3}=b_{2,4}=b_{2,5}=b_{2,6}=b_{2,7}=b_{2,8}
\]
\[
=b_{3,2}=b_{3,4}=b_{3,5}=b_{3,6}=b_{3,7}=b_{3,8}
\]
\[
=b_{4,1}=b_{4,2}=b_{4,5}=b_{4,6}=b_{4,7}=b_{4,8}
\]
\[
=b_{5,1}=b_{5,2}=b_{5,3}=b_{5,4}=b_{5,5}=b_{5,6}=b_{5,7}=b_{5,8}
\]
\[
=b_{6,1}=b_{6,2}=b_{6,3}=b_{6,4}=b_{6,7}=b_{6,8}
\]
\[
=b_{7,1}=b_{7,2}=b_{7,3}=b_{7,4}=b_{7,5}=b_{7,6}
\]
\[
=b_{8,1}=b_{8,2}=b_{8,3}=b_{8,4}=b_{8,5}=b_{8,6}=b_{8,7}=b_{8,8}=0.
\]
Then for every $x\in W(2)$ there exist elements $a_x$, $b_x$ in ${ \mathbb F}$ such that
\begin{equation} \label{c}
\left\{
  \begin{array}{ll}
    b_{1,2}x_2=a_xx_2 \hbox{;} \\
    b_{2,2}x_2=-b_xx_2 \hbox{;} \\
    b_{3,1}x_1+b_{3,3}x_3=2a_xx_1+b_xx_3 \hbox{;} \\
    b_{4,3}x_3+b_{4,4}x_4=3a_xx_3+2b_xx_4 \hbox{;} \\
    b_{6,5}x_5+b_{6,6}x_6=-a_xx_5+b_xx_6 \hbox{;} \\
    b_{7,7}x_7+b_{7,8}x_8=b_xx_7+a_xx_8 \hbox{.}
  \end{array}
\right.
\end{equation}

Using $1$-th and $3$-th equalities of system (\ref{c}) we get
\[
\left\{
  \begin{array}{ll}
2b_{1,2} x_1x_2 = 2a_x x_1x_2 \hbox{;} \\
b_{3,1} x_1x_2+b_{3,3} x_2x_3 = 2a_x x_1x_2+b_x x_2x_3 \hbox{.}
  \end{array}
\right.
\]
and
\[
(b_{3,1}-2b_{1,2})x_1x_2+b_{3,3}x_2x_3=b_xx_2x_3.
\]
Hence, $b_{3,1}=2b_{1,2}$. Similarly, using equalities of (\ref{c}) we get
\[
b_{4,3}=3b_{1,2}, b_{2,2}=-b_{3,3}, b_{4,4}=-2b_{2,2}.
\]
Using $1$-th and $5$-th equalities of system (\ref{c}) we get
\[
\left\{
  \begin{array}{ll}
b_{1,2} x_2x_5 = a_x x_2x_5 \hbox{;} \\
b_{6,5} x_5x_2+b_{6,6} x_6x_2 = -a_x x_5x_2+b_x x_6x_2 \hbox{.}
  \end{array}
\right.
\]
and
\[
(b_{6,5}+b_{1,2})x_2x_5+b_{6,6}x_6x_2=b_xx_6x_2.
\]
Hence, $b_{6,5}=-b_{1,2}$.

Using $2$-th and $5$-th equalities of system (\ref{c}) we get
\[
\left\{
  \begin{array}{ll}
b_{2,2} x_2x_6 = -b_x x_2x_6 \hbox{;} \\
b_{6,5} x_5x_2+b_{6,6} x_6x_2 = -a_x x_5x_2+b_x x_6x_2 \hbox{.}
  \end{array}
\right.
\]
and
\[
b_{6,5} x_5x_2+(b_{6,6}+b_{2,2})x_6x_2=-a_xx_5x_2.
\]
Hence, $b_{6,6}=-b_{2,2}$.

Using $1$-th and $6$-th equalities of system (\ref{c}) we get
\[
\left\{
  \begin{array}{ll}
b_{1,2} x_2x_8 = a_x x_2x_8 \hbox{;} \\
b_{7,7} x_7x_2+b_{7,8} x_8x_2 = b_x x_7x_2+a_x x_8x_2 \hbox{.}
  \end{array}
\right.
\]
and
\[
b_{7,7}x_7x_2+(b_{7,8}-b_{1,2})x_8x_2=b_xx_7x_2.
\]
Hence, $b_{7,8}=b_{1,2}$.

Using $2$-th and $6$-th equalities of system (\ref{c}) we get
\[
\left\{
  \begin{array}{ll}
b_{2,2} x_2x_7 = -b_x x_2x_7 \hbox{;} \\
b_{7,7} x_7x_2+b_{7,8} x_8x_2 = b_x x_7x_2+a_x x_8x_2 \hbox{.}
  \end{array}
\right.
\]
and
\[
(b_{7,7}+b_{2,2})x_7x_2+b_{7,8}x_8x_2=a_xx_8x_2.
\]
Hence, $b_{7,7}=-b_{2,2}$.

These equalities show that the matrix of the linear map $\nabla$
is of the form (\ref{(1.3)}). Therefore, by lemma \ref{3}
$\nabla$ is a derivation. This completes the proof.
\end{proof}

Since a derivation on $W(2)$ is invariant on the subalgebras $S_2$ and $W_2$, we have
the following corollary.

\begin{corollary}
Every local derivation of the algebras $S_2$ and $W_2$ is a derivation.
\end{corollary}

\section{$2$-Local derivations of conservative algebras of $2$-dimensional algebras}

In this section we give another characterization of derivations on conservative algebras of $2$-dimensional algebras.

A (not necessary linear) map $\Delta : \mathcal{A}\to \mathcal{A}$ is called a $2$-local derivation, if for
any elements $x$, $y\in \mathcal{A}$ there exists a derivation
$D_{x,y}:\mathcal{A}\to \mathcal{A}$ such that $\Delta (x)=D_{x,y}(x)$, $\Delta (y)=D_{x,y}(y)$.

\begin{theorem} \label{6}
Every $2$-local derivation of the algebras $S_2$, $W_2$ and $W(2)$ is a derivation.
\end{theorem}

\begin{proof}
We will prove that every $2$-local derivation of $W(2)$ is a derivation.

Let $\Delta$ be an arbitrary $2$-local derivation of $W(2)$.
T   hen, by the definition, for every element $a\in W(2)$,
there exists a derivation $D_{a,e_2}$ of $W(2)$ such that
$$
\Delta(a)=D_{a,e_2}(a), \,\,\,  \Delta(e_2)=D_{a,e_2}(e_2).
$$
By lemma \ref{3}, the matrix $A^{a,e_2}$ of the derivation $D_{a,e_2}$ has the following matrix form:
\[
A^{a,e_2}=
\left(
    \begin{array}{cccccccc}
      0 & \alpha_{a,e_2} & 0 & 0 & 0 & 0 & 0 & 0\\
      0 & -\beta_{a,e_2} & 0 & 0 & 0 & 0 & 0 & 0\\
      2\alpha_{a,e_2} & 0 & \beta_{a,e_2} & 0 & 0 & 0 & 0 & 0\\
      0 & 0 & 3\alpha_{a,e_2} & 2\beta_{a,e_2} & 0 & 0 & 0 & 0 \\
      0 & 0 & 0 & 0 & 0 & 0 & 0 & 0\\
      0 & 0 & 0 & 0 & -\alpha_{a,e_2} & \beta_{a,e_2} & 0 & 0\\
      0 & 0 & 0 & 0 & 0 & 0 & \beta_{a,e_2} & \alpha_{a,e_2}\\
      0 & 0 & 0 & 0 & 0 & 0 & 0 & 0\\
      \end{array}
  \right).
\]

Let $v$ be an arbitrary element in $W(2)$. Then
there exists a derivation $D_{v,e_2}$ of $W(2)$ such that
$$
\Delta(v)=D_{v,e_2}(v), \,\,\,  \Delta(e_2)=D_{v,e_2}(e_2).
$$
By lemma \ref{3}, the matrix $A^{v,e_2}$ of the derivation $D_{v,e_2}$ has the following matrix form:
\[
A^{v,e_2}=
\left(
    \begin{array}{cccccccc}
      0 & \alpha_{v,e_2} & 0 & 0 & 0 & 0 & 0 & 0\\
      0 & -\beta_{v,e_2} & 0 & 0 & 0 & 0 & 0 & 0\\
      2\alpha_{v,e_2} & 0 & \beta_{v,e_2} & 0 & 0 & 0 & 0 & 0\\
      0 & 0 & 3\alpha_{v,e_2} & 2\beta_{v,e_2} & 0 & 0 & 0 & 0 \\
      0 & 0 & 0 & 0 & 0 & 0 & 0 & 0\\
      0 & 0 & 0 & 0 & -\alpha_{v,e_2} & \beta_{v,e_2} & 0 & 0\\
      0 & 0 & 0 & 0 & 0 & 0 & \beta_{v,e_2} & \alpha_{v,e_2}\\
      0 & 0 & 0 & 0 & 0 & 0 & 0 & 0\\
      \end{array}
  \right).
\]
Since $\Delta(e_2)=D_{a,e_2}(e_2)=D_{v,e_2}(e_2)$, we have
\[
\alpha_{a,e_2}=\alpha_{v,e_2}, \beta_{a,e_2}=\beta_{v,e_2},
\]
that it
\[
D_{v,e_2}=D_{a,e_2}.
\]
Therefore, for any element $a$ of the algebra $W(2)$
\[
\Delta(a)=D_{v,e_2}(a),
\]
that it $D_{v,e_2}$ does not depend on $a$. Hence, $\Delta$
is a derivation by lemma \ref{(1.3)}.

The cases of the algebras $S_2$ and $W_2$ are also similarly proved.
This ends the proof.
\end{proof}

\section{2-Local automorphisms of conservative algebras of 2-dimensional algebras}

A (not necessary linear) map $\Delta : \mathcal{A}\to \mathcal{A}$ is called a $2$-local automorphism, if for
any elements $x$, $y\in \mathcal{A}$ there exists an automorphism
$\phi_{x,y}:\mathcal{A}\to \mathcal{A}$ such that $\Delta (x)=\phi_{x,y}(x)$, $\Delta (y)=\phi_{x,y}(y)$.

\begin{theorem} \label{4}
Every $2$-local automorphism of the algebras $S_2$, $W_2$ and $W(2)$ is an automorphism.
\end{theorem}

\begin{proof}
We prove that every $2$-local automorphism of $W(2)$ is an automorphism.

Let $\Delta$ be an arbitrary $2$-local automorphism of $W(2)$.
Then, by the definition, for every element $x\in W(2)$,
\[
x=x_1e_1+x_2e_2+x_3e_3+x_4e_4+x_5e_5+x_6e_6+x_7e_7+x_8e_8,
\]
there exist elements $a_{x,e_2}$, $b_{x,e_2}$ such that
\[
A_{x,e_2}=
\left(
    \begin{array}{cccccccc}
      1 & a_{x,e_2} & 0 & 0 & 0 & 0 & 0 & 0\\
      0 & \frac{1}{b_{x,e_2}} & 0 & 0 & 0 & 0 & 0 & 0\\
      2a_{x,e_2}b_{x,e_2} & a_{x,e_2}^2b_{x,e_2} & b_{x,e_2} & 0 & 0 & 0 & 0 & 0\\
      3a_{x,e_2}^2b_{x,e_2}^2 & a_{x,e_2}^3b_{x,e_2}^2 & 3a_{x,e_2}b_{x,e_2}^2 & b_{x,e_2}^2 & 0 & 0 & 0 & 0 \\
      0 & 0 & 0 & 0 & 1 & 0 & 0 & 0\\
      0 & 0 & 0 & 0 & -a_{x,e_2}b_{x,e_2} & b_{x,e_2} & 0 & 0\\
      0 & 0 & 0 & 0 & 0 & 0 & b_{x,e_2} & a_{x,e_2}b_{x,e_2}\\
      0 & 0 & 0 & 0 & 0 & 0 & 0 & 1\\
      \end{array}
  \right),
\]
$\Delta(x)=A_{x,e_2}\bar{x}$, where $\bar{x}=(x_1,x_2,x_3,x_4,x_5,x_6,x_7,x_8)$ is the vector corresponding to $x$, and
\[
\Delta(e_2)=A_{x,e_2}e_2=
(a_{x,e_2}, \frac{1}{b_{x,e_2}}, a_{x,e_2}^2b_{x,e_2},  a_{x,e_2}^3b_{x,e_2}^2, 0, 0, 0, 0).
\]
Since the element $x$ was chosen arbitrarily, we have
\[
\Delta(e_2)=(a_{x,e_2}, \frac{1}{b_{x,e_2}}, a_{x,e_2}^2b_{x,e_2}, a_{x,e_2}^3b_{x,e_2}^2, 0, 0, 0, 0)
\]
\[
=(a_{y,e_2}, \frac{1}{b_{y,e_2}}, a_{y,e_2}^2b_{y,e_2}, a_{y,e_2}^3b_{y,e_2}^2, 0, 0, 0, 0),
\]
for each pair $x$, $y$ of elements in $W(2)$.
Hence, $a_{x,e_2}=a_{y,e_2}$, $b_{x,e_2}=b_{y,e_2}$.
Therefore
\[
\Delta(x)=A_{y,e_2}x
\]
for any $x\in W(2)$ and the matrix $A_{y,e_2}$ does not depend on $x$.
Thus, by lemma \ref{21} $\Delta$ is an automorphism.

The cases of the algebras $S_2$ and $W_2$ are also similarly proved.
The proof is complete.
\end{proof}

\section{Local automorphisms of conservative algebras of 2-dimensional algebras}

Let $\mathcal{A}$ be an algebra.
A linear map $\nabla : \mathcal{A}\to \mathcal{A}$ is called a local automorphism, if for
any element $x\in \mathcal{A}$ there exists an automorphism
$\phi_x:\mathcal{A}\to \mathcal{A}$ such that $\nabla(x)=\phi_x(x)$.

\begin{theorem}  \label{5}
Every local automorphism of the algebras $S_2$, $W_2$ and $W(2)$ is an automorphism.
\end{theorem}

\begin{proof}
We prove that every local automorphism of $W(2)$ is an automorphism.

Let $\nabla$ be an arbitrary local automorphism of $W(2)$ and $B$ be its matrix, i.e.,
\[
\nabla(x)=B\bar{x}, x\in W(2),
\]
where $\bar{x}$ is the vector corresponding to $x$.
Then, by the definition, for every element $x\in W(2)$,
\[
x=x_1e_1+x_2e_2+x_3e_3+x_4e_4+x_5e_5+x_6e_6+x_7e_7+x_8e_8,
\]
there exist elements $a_x$, $b_x$ such that
\[
A_x=
\left(
    \begin{array}{cccccccc}
      1 & a_x & 0 & 0 & 0 & 0 & 0 & 0\\
      0 & \frac{1}{b_x} & 0 & 0 & 0 & 0 & 0 & 0\\
      2a_xb_x & a_x^2b_x & b_x & 0 & 0 & 0 & 0 & 0\\
      3a_x^2b_x^2 & a_x^3b_x^2 & 3a_xb_x^2 & b_x^2 & 0 & 0 & 0 & 0 \\
      0 & 0 & 0 & 0 & 1 & 0 & 0 & 0\\
      0 & 0 & 0 & 0 & -a_xb_x & b_x & 0 & 0\\
      0 & 0 & 0 & 0 & 0 & 0 & b_x & a_xb_x\\
      0 & 0 & 0 & 0 & 0 & 0 & 0 & 1\\
      \end{array}
  \right)
\]
and
\[
\nabla(x)=B\bar{x}=A_x\bar{x}.
\]
Using these equalities and by choosing subsequently $x=e_1$, $x=e_2$, $\dots, x=e_8$ we get
\[
B=
\left(
    \begin{array}{cccccccc}
      1 & a_{e_2} & 0 & 0 & 0 & 0 & 0 & 0\\
      0 & \frac{1}{b_{e_2}} & 0 & 0 & 0 & 0 & 0 & 0\\
      2a_{e_1}b_{e_1} & a_{e_2}^2b_{e_2} & b_{e_3} & 0 & 0 & 0 & 0 & 0\\
      3a_{e_1}^2b_{e_1}^2 & a_{e_2}^3b_{e_2}^2 & 3a_{e_3}b_{e_3}^2 & b_{e_4}^2 & 0 & 0 & 0 & 0 \\
      0 & 0 & 0 & 0 & 1 & 0 & 0 & 0\\
      0 & 0 & 0 & 0 & -a_{e_5}b_{e_5} & b_{e_6} & 0 & 0\\
      0 & 0 & 0 & 0 & 0 & 0 & b_{e_7} & a_{e_8}b_{e_8}\\
      0 & 0 & 0 & 0 & 0 & 0 & 0 & 1\\
      \end{array}
  \right).
\]
Since $\nabla(e_6+e_7)=\nabla(e_6)+\nabla(e_7)$, we have
$$
b_{e_6+e_7}=b_{e_6}, b_{e_6+e_7}=b_{e_7}.
$$
Hence,
$$
b_{e_6}=b_{e_7}.
$$
Similarly to this equality we get $b_{e_3}=b_{e_6}$ and $b_{e_6}=b_{e_2}\neq 0$.
Hence,
\begin{equation} \label{41}
b_{e_2}=b_{e_3}=b_{e_6}=b_{e_7}.
\end{equation}

Since $\nabla(e_5+e_8)=\nabla(e_5)+\nabla(e_8)$, we have
$$
a_{e_5+e_8}b_{e_5+e_8}=a_{e_5}b_{e_5}, \,\, a_{e_5+e_8}b_{e_5+e_8}=a_{e_8}b_{e_8}.
$$
From this it follows that
$$
a_{e_5}b_{e_5}=a_{e_8}b_{e_8}.
$$
Similarly to this equality we get $a_{e_1}b_{e_1}=a_{e_8}b_{e_8}$.
Hence,
\begin{equation} \label{42}
a_{e_1}b_{e_1}=a_{e_5}b_{e_5}=a_{e_8}b_{e_8}.
\end{equation}

Since $\nabla(e_4+e_6)=\nabla(e_4)+\nabla(e_6)$, we have
\[
b_{e_4+e_6}^2=b_{e_4}^2, \,\, b_{e_4+e_6}^2=b_{e_6}^2.
\]
From this it follows that
\[
b_{e_4}^2=b_{e_6}^2.
\]
Hence, by (\ref{41}), we get
\begin{equation} \label{43}
b_{e_4}^2=b_{e_2}^2.
\end{equation}

Since $\nabla(e_2+e_8)=\nabla(e_2)+\nabla(e_8)$, we have
\[
a_{e_2}=a_{e_2+e_8}, \,\,\, a_{e_2+e_8}^2b_{e_2+e_8}=a_{e_2}^2b_{e_2}, \,\,\, a_{e_2+e_8}b_{e_2+e_8}=a_{e_8}b_{e_8}.
\]
Hence,
\[
b_{e_2+e_8}=b_{e_2}, \,\,\, a_{e_2+e_8}b_{e_2+e_8}=a_{e_2}b_{e_2}
\]
and, therefore,
\begin{equation} \label{44}
a_{e_2}b_{e_2}=a_{e_8}b_{e_8}.
\end{equation}

Similarly, since $\nabla(e_2+e_3)=\nabla(e_2)+\nabla(e_3)$, we have
\[
a_{e_2}=a_{e_2+e_3}, \,\,\, b_{e_2}^{-1}=b_{e_2+e_3}^{-1}, \,\,\,
a_{e_2+e_3}^3b_{e_2+e_3}^2+3a_{e_2+e_3}b_{e_2+e_3}^2=a_{e_2}^3b_{e_2}^2+3a_{e_3}b_{e_3}^2.
\]
Hence,
\[
b_{e_2}=b_{e_2+e_3}
\]
and by (\ref{41}) and $a_{e_2}=a_{e_2+e_3}$ we get
\[
a_{e_2}^3+3a_{e_2}=a_{e_2}^3+3a_{e_3}.
\]
Therefore, $a_{e_2}=a_{e_3}$ and
\begin{equation} \label{45}
a_{e_2}b_{e_2}^2=a_{e_3}b_{e_3}^2.
\end{equation}

Finally, since $\nabla(e_1+e_8)=\nabla(e_1)+\nabla(e_8)$, we have
$$
a_{e_1+e_8}b_{e_1+e_8}=a_{e_1}b_{e_1}, \,\,\, a_{e_1+e_8}b_{e_1+e_8}=a_{e_8}b_{e_8}.
$$
Hence,
$$
a_{e_1}b_{e_1}=a_{e_8}b_{e_8}.
$$
By (\ref{44}), from the last equalities it follows that
\begin{equation}  \label{46}
a_{e_1}b_{e_1}=a_{e_2}b_{e_2}, \,\,\, a_{e_1}^2b_{e_1}^2=(a_{e_1}b_{e_1})^2=(a_{e_2}b_{e_2})^2=a_{e_2}^2b_{e_2}^2.
\end{equation}

By (\ref{41}), (\ref{42}), (\ref{43}), (\ref{44}), (\ref{45}), (\ref{46}) the matrix $B$ has the following matrix form
\[
B=
\left(
    \begin{array}{cccccccc}
      1 & a_{e_2} & 0 & 0 & 0 & 0 & 0 & 0\\
      0 & \frac{1}{b_{e_2}} & 0 & 0 & 0 & 0 & 0 & 0\\
      2a_{e_2}b_{e_2} & a_{e_2}^2b_{e_2} & b_{e_2} & 0 & 0 & 0 & 0 & 0\\
      3a_{e_2}^2b_{e_2}^2 & a_{e_2}^3b_{e_2}^2 & 3a_{e_2}b_{e_2}^2 & b_{e_2}^2 & 0 & 0 & 0 & 0 \\
      0 & 0 & 0 & 0 & 1 & 0 & 0 & 0\\
      0 & 0 & 0 & 0 & -a_{e_2}b_{e_2} & b_{e_2} & 0 & 0\\
      0 & 0 & 0 & 0 & 0 & 0 & b_{e_2} & a_{e_2}b_{e_2}\\
      0 & 0 & 0 & 0 & 0 & 0 & 0 & 1\\
      \end{array}
  \right).
\]
Hence, by lemma \ref{21}, the local automorphism $\nabla$ is an automorphism.

The cases of the algebras $S_2$ and $W_2$ are also similarly proved.
This ends the proof.
\end{proof}

The authors thank professor Ivan Kaygorodov for detailed reading of this work and for
suggestions which improved the paper.

\end{document}